\documentclass[12pt,a4paper]{amsart}

\usepackage[numbers]{natbib}

\usepackage{enumitem,kantlipsum}
\usepackage[english]{babel}
\usepackage[utf8]{inputenc}

\usepackage{ifxetex}
\usepackage{pb-diagram}
\usepackage{amsfonts}
\usepackage{amsmath, amsthm, amscd, amsfonts, xpatch}
\usepackage{mathtools}
\usepackage{mathrsfs}
\usepackage{latexsym}
\usepackage{amssymb}
\usepackage{caption}

\usepackage{graphicx,float}
\usepackage{tikz}
\usepackage{pst-node}
\usepackage{tikz-cd} 

\usepackage{setspace}
\definecolor{persianblue}{rgb}{0.11, 0.22, 0.73}
\definecolor{persiangreen}{rgb}{0.0, 0.65, 0.58}

\usepackage{lineno}
\usepackage[hidelinks]{hyperref}
\hypersetup{
  colorlinks   = true, 
  urlcolor     = persiangreen, 
  linkcolor    = persianblue, 
  citecolor   = persianblue 
}

\usepackage{bussproofs}
\usepackage{csquotes}

\DeclareMathAlphabet{\mathpzc}{OT1}{pzc}{m}{it}
\usepackage{afterpage}

\usepackage[capitalise]{cleveref}

\input xy
\xyoption{all}

\newtheorem{theorem}{Theorem}[section]
\newtheorem{lemma}[theorem]{Lemma}

\newtheorem{proposition}[theorem]{Proposition}
\newtheorem{corollary}[theorem]{Corollary}

\newtheorem{definition}[theorem]{Definition}

\numberwithin{equation}{section}

\newcommand{\rest}{\! \upharpoonright \!}

\theoremstyle{remark}

\makeatletter
\let\qed@empty\openbox 
\def\@begintheorem#1#2[#3]{%
  \deferred@thm@head{%
    \the\thm@headfont\thm@indent
    \@ifempty{#1}
      {\let\thmname\@gobble}
      {\let\thmname\@iden}%
    \@ifempty{#2}
      {\let\thmnumber\@gobble\global\let\qed@current\qed@empty}
      {\let\thmnumber\@iden\xdef\qed@current{#2}}%
    \@ifempty{#3}
      {\let\thmnote\@gobble}
      {\let\thmnote\@iden}%
    \thm@swap\swappedhead
    \thmhead{#1}{#2}{#3}%
    \the\thm@headpunct\thmheadnl\hskip\thm@headsep
  }\ignorespaces
}
\renewcommand{\qedsymbol}{%
  \ifx\qed@thiscurrent\qed@empty
    \qed@empty
  \else
    \fbox{\scriptsize\qed@thiscurrent}%
  \fi
}
\renewcommand{\proofname}{%
  Proof%
  \ifx\qed@thiscurrent\qed@empty
  \else
    \ of \qed@thiscurrent
  \fi
}
\xpretocmd{\proof}{\let\qed@thiscurrent\qed@current}{}{}
\newenvironment{proof*}[1]
  {\def\qed@thiscurrent{\ref{#1}}\proof}
  {\endproof}
\makeatother
\begin{document}


\title[Almost Strong Properness]{Almost Strong Properness} 

\author[R. Mohammadpour]{Rahman Mohammadpour}

\email{\href{mailto:rahmanmohammadpour@gmail.com}{rahmanmohammadpour@gmail.com}}
\urladdr{\href{https://sites.google.com/site/rahmanmohammadpour}{https://sites.google.com/site/rahmanmohammadpour}}



\subjclass[2010]{03E05, 03E35, 03E57} 
\keywords{MRP, PFA, Side condition, Suslin tree}

\begin{abstract} 
We introduce the forcing property ``almost strong properness'' which sits between properness and strong properness. As an application, we introduce a simple forcing with finite conditions to force $\rm MRP$.
\end{abstract}

\maketitle

\section*{Introduction}
The Mapping Reflection Principle ($\rm MRP$) was discovered and shown to be a consequence of $\rm PFA$ by Moore \cite{MooreMRP} using a proper forcing with countable conditions.
This is a strong reflection principle which decides the value of the continuum to be $\aleph_2$, and implies, among other things, both  the Singular Cardinal Hypothesis and the failure of  square principle, see \cite{VaileSCH}.
 In this note, we introduce a  subclass of proper forcing notions which contains strongly proper forcings, we shall observe that almost strongly proper forcings preserve c.c.c-ness and have the $\omega_1$-approximation property, and hence they preserve Suslin trees. We then demonstrate how to force $\rm MRP$ using an almost strongly proper forcing with finite conditions. Consequently,  $\rm MRP$ holds under Todorčević's forcing axiom $\rm PFA(S)$. This fact was first proved by Teruyuki Yorioka and Tadatoshi Miyamoto in \cite{MY}. In the same paper, they introduced also the notion of an almost strongly proper forcing for another purpose, but it is apparently  different from ours.

\section{ Almost Strong Properness}
Let us  recall the definition of a strongly generic condition due to Mitchell \cite{MI2005}.
Suppose $\mathbb P$ is a forcing, and that $X$ is a set. A condition
$p\in\mathbb P$ is called $(X,\mathbb P)$- strongly generic if and only if for every $q\leq p$ there is
$q\rest_X\in X\cap \mathbb P$ such that every $r\in X\cap\mathbb P$ with $r\leq q\rest_X$ is compatible with $q$. The notion of strong properness is defined in an obvious way.
\begin{definition}[Almost Strong Genericity]
Suppose $\mathbb P$ is a forcing, and $M\prec H_\theta$ contains $[\mathbb P]^{\omega}$. A condition
$p\in\mathbb P$ is called \emph{ $(M,\mathbb P)$-almost strongly generic} if and only if for every $q\leq p$,
$$S_q\coloneqq\{X\in [\mathbb P]^\omega: \exists q\rest_X\in X \text{ such that } \forall r\in X, r\leq q\rest_X\Rightarrow r||q \}$$
is $M$-stationary, i.e., for every algebra\footnote{Recall that an algebra over a set $Z$ is a function from $[Z]^{<\omega}$ to $Z$.} $F\in M$ over $\mathbb P$, there is some 
$A\in M\cap S_q$ which is $F$-closed. 
\end{definition}
It is clear that every $(M,\mathbb P)$-strongly generic condition is $(M,\mathbb P)$-almost strongly generic, and that every  $(M,\mathbb P)$-almost strongly generic condition is $(M,\mathbb P)$-generic. 

\begin{definition}[Almost Strong Properness]
A forcing notion $\mathbb P$ is called \emph{almost strongly proper} (a.s.p. for short), if for every sufficiently large regular cardinal $\theta$, there is a club of countable models $M\prec H_\theta$ containing $[\mathbb P]^\omega$ such that every condition $p\in M\cap\mathbb P$ can be extended to an $(M,\mathbb P)$-almost strongly generic condition.
\end{definition}

The following easy lemma appears frequently in applications.
\begin{lemma}\label{unboundedness}
Assume that $p\in\mathbb P$. Let $M\prec H_\theta$ with $[\mathbb P]^{\omega}\in M$, where $\theta$ is a sufficiently large regular cardinal. Then $S_p$ is $M$-stationary if and only if 
 for every $U\in M$ which is unbounded in $[\mathbb P]^\omega$, we have that $U\cap S_p\cap M\neq\varnothing$.
\end{lemma}
\begin{proof}
Suppose $U\in M$ is an unbounded subset of $[\mathbb P]^\omega$. Let $\overline{U}$ be defined as follows:
$$\overline{U}\coloneqq\{X\in [\mathbb P]^\omega: \forall s\in [X]^{<\omega}~\exists u\in U \text{ such that } s\subset u\subseteq X\}.$$
It is obvious that $\overline{U}$ is a club in $[\mathbb P]^\omega$  containing $U$ as a subset. Moreover, $\overline{U}$ is in $M$. Since $S_p$ is  $M$-stationary, and $\overline{U}$ is in $M$, there is some $X\in \overline{U}\cap M\cap S_p$. Thus there is $p\rest_X\in X$ with the property that all its extensions in $X$ are compatible with $p$. Since $X\in\overline{U}\cap M$ and $p\rest_X\in X\cap M$, by elementarity and  definition of $\overline{U}$,  there is some $u\in U\cap M$ such that
$\{p\rest_X\}\subset u\subseteq X$.
Obviously, every extension of $p\rest_X$ in $u$ is compatible with $p$. Consequently,  $ U\cap S_p\cap M \neq\varnothing$. The other direction is trivial.
\end{proof}

Recall that a forcing notion $\mathbb P$ has the $\omega_1$-approximation property if for every $V$-generic filter $G\subseteq\mathbb P$, the pair $(V,V[G])$ has the $\omega_1$-approximation property, i.e., every set of ordinals $x\in V[G]$ such that $x\cap a\in V$ for every countable $a\in V$, should be  in $V$ (see \cite{Hamkins2003}).  The following is well-known and easy to prove.

\begin{lemma}\label{no new branch}
No forcing  with the $\omega_1$-approximation property can add
new cofinal branches through  trees of height $\omega_1$.
\end{lemma}
\begin{proof}[\nopunct]

\end{proof}

\begin{proposition}\label{asp approx}
Every a.s.p. forcing has the $\omega_1$-approximation property.
\end{proposition} 
\begin{proof}
Assume that $\mathbb P$ is an a.s.p. forcing. Suppose that $p_0\in\mathbb P$ forces  $\dot{f}:\gamma \rightarrow 2 $ to be a countably approximated function, where $\gamma$ is an ordinal. We shall show that the set of conditions  deciding $\dot f$ is dense below $p_0$. Fix $p\leq p_0$. Suppose that $\theta<\theta^*$ are sufficiently large regular cardinals. Let $M\prec H_{\theta^*}$ be countable and contain the relevant objects, in particular $p$. Let $q\leq p$ be an $(M,\mathbb P)$-almost strongly generic condition. One can extend $q$ further to make sure that $q\Vdash \dot f\rest_M\in V$. Therefore, We may assume without loss of generality that there is, in $V$, some function $g:M\cap\gamma\rightarrow 2$ such that $q\Vdash g=\dot f\rest_M$. 
Let 
$$\mathcal U=\{N\cap \mathbb P: \mathbb P,\gamma,\dot{f}\in N\prec H_\theta \text{ and  }N \text{ is countable}\}.$$

$\mathcal U$ is unbounded in $[\mathbb P]^\omega$ and belongs to $M$. We now use the almost strong genericity of $q$ and \cref{unboundedness} to pick some $N\prec H_\theta$ in $M$ with $\dot{f},\gamma,\mathbb P\in N$ for which there exists $q\rest_N$ in $N\cap\mathbb P$ so that every condition in $N$ extending $q\rest_N$ is compatible with $q$. We are done if $q\rest_N$
 decides $\dot f$. Suppose this is not the case thus there are, by elementarity, $\zeta\in\gamma\cap N$ (and hence in $M$)
 and $q_0,q_1\in N$ extending $q\rest_N$ such that
 $q_0\Vdash \dot f(\zeta)=0$ and $q_1\Vdash \dot f(\zeta)=1$, but this is impossible since $q\Vdash \dot f(\zeta)=g(\zeta)$; otherwise  either $q_0$ or $q_1$ is incompatible with $q$, a contradiction!
 \end{proof}

\begin{proposition}\label{asp ccc}
Every a.s.p. forcing preserves c.c.c-ness.
\end{proposition}
\begin{proof}
Suppose  $\mathbb Q$ is a c.c.c forcing, and that $\mathbb P$ is an a.s.p. forcing. Choose  sufficiently large regular cardinals $\theta<\theta^*$ with $\mathbb P,\mathbb Q\in H_{\theta}$. Suppose $\mathcal C$ is a club in $[H_{\theta^*}]^\omega$  witnessing the almost strong properness of $\mathbb P$. Let $M\prec H_{\theta^*}$ in $\mathcal C$ contain
$\mathbb P$ and $\mathbb Q$. It is enough to show that for every $q\in \mathbb Q$ and every $(M,\mathbb P)$-almost strongly generic condition $p\in\mathbb P$, $(p,q)$ is $(M,\mathbb P\times\mathbb \mathbb Q)$-generic\footnote{Recall that a forcing notion is c.c.c if and only if the maximal condition is generic for unboundedly many elementary submodel in some $H_\theta$ big enough.}. Moreover, 
to see this, it is enough to show that whenever $D\in M$ is a dense subset of $\mathbb P\times\mathbb Q$,  $(p,q)$ is compatible with some condition in $D\cap M$. Thus
let $D\in M$ be a dense subset of $\mathbb P\times\mathbb Q$. Consider
$$\mathcal U=\{N\cap \mathbb P: \mathbb P,\mathbb Q,D\in N\prec H_\theta\text{ and  }N \text{ is countable}\}.$$

$\mathcal U$ is unbounded in $[\mathbb P]^\omega$ and belongs to $M$. By the almost strong genericity of $p$ and \cref{unboundedness}, one can choose $N\prec H_\theta$ in $M$ with $\mathbb P,\mathbb Q,D\in N$  so that there is $p\rest_N\in N\cap\mathbb P$, for which every stronger condition in $N\cap\mathbb P$ is compatible with $p$.
Set
$$E=\{q'\in\mathbb Q: \exists p'\leq p\restriction_N \text{ with } (p',q')\in D \}.$$
 It is easily seen that $E$ is a dense subset of
$\mathbb Q$. Clearly $E$  is in $N$, and hence in $M$. Now since $\mathbb Q$ is c.c.c., there is $q'\in N\cap E$ such that $q'$ is compatible with $q$. By the definition of $E$ and the elementarity of $N$, there is $p'\leq p\rest_N$ in $N$ so that
$(p',q')\in D\cap N$, which in turn implies that $(p',q') $ is compatible with $(p,q)$. Notice that
$(p',q')\in N\cap D\subseteq M\cap D$.

\end{proof}

\begin{corollary}
Every a.s.p. forcing preserves Suslinity.
\end{corollary}
\begin{proof}
Assume that $\mathbb P$ is an a.s.p. forcing.  Suppose $S$ is a Suslin tree. Let $G\subseteq \mathbb P$ be $V$-generic filter. By \cref{no new branch} and \cref{asp approx}, $S$ does not have cofinal branches in $V[G]$. Now,   consider the corresponding Suslin forcing of $S$, say $\mathbb S$ which is  c.c.c.  By \cref{asp ccc}, $\mathbb S$ remains c.c.c., and hence $S$ is a  Suslin tree in $V[G]$.
\end{proof}

\section{Mapping Reflection Principle}

 \begin{definition}[Ellentuck Topology]
Let $X$ be an uncountable set. The \emph{Ellentuck topology} on $[X]^{\leq\omega}$ is the topology generated by the following sets as basic open sets.
$$[a,A]\coloneqq\{x\subseteq X: a\subseteq x\subseteq A\}, \text{ where } a \text{ is finite and } A\subseteq X \text{ is countable }.$$
\end{definition}
The following definitions are due to Moore, \cite{MooreMRP}.
\begin{definition}[$M$-stationarity]
Suppose $M$ and $X$ are  sets. A set $\Sigma\subseteq [X]^{\leq\omega}$ 
 is called \emph{$M$-stationary} if for every algebra $F\in M$ over $X$, there is some countable set
$A\in M\cap \Sigma$ closed under $F$.
\end{definition}

\begin{definition}[Open and Stationary Mapping]
 A function $\Sigma$ is called \emph{open and stationary mapping} if there are a regular cardinal 
 $\theta=\theta_\Sigma$ and an uncountable set $X=X_\Sigma$ with $[X]^{\omega}\in H_\theta$ such that:
\begin{enumerate}
\item ${\rm dom}(\Sigma)$ is the collection of countable elementary submodels of $H_\theta$ containing $[X]^{\omega}$. 
\item For each $M\in{\rm dom}(\Sigma)$, $\Sigma(M)$ is  $M$-stationary, and also open in $[X]^{\leq\omega}$
 with respect to the Ellentuck topology.
\end{enumerate}

\end{definition}
\begin{definition}[Reflection]
An open stationary mapping $\Sigma$  \emph{reflects} if there is a continuous $\in$-chain $\langle M_\xi:\xi<\omega_1\rangle$
of models in ${\rm dom}(\Sigma)$ such that for every non-zero limit ordinal $\xi<\omega_1$, there is $\zeta<\xi$ so that for every $\eta\in\xi\setminus \zeta$,
$M_\eta\cap X \in \Sigma(M_\xi)$. The sequence $\langle M_\xi:\xi<\omega_1\rangle$ is called a \emph{reflecting sequence} for $\Sigma$.
\end{definition}

\begin{definition}[$\rm MRP$]\label{MRP}
The Mapping Reflection Principle $(\rm MRP)$ states that every open stationary mapping reflects.
\end{definition}

The following can be proved in the same way as \cref{unboundedness}.
\begin{lemma}\label{unboundedness 2}
Suppose $\Sigma\subseteq [X]^{\leq \omega}$ is open. Then $\Sigma$ is $M$-stationary if and only if for every $U\in M$ which is unbounded in $[X]^\omega$, we have that $U\cap \Sigma\cap M\neq\varnothing$.
\end{lemma}
\begin{proof}[\nopunct]

\end{proof}
\begin{theorem}\label{main theorem}
 Suppose that $\Sigma$ is an open stationary mapping. Then there is an a.s.p.   forcing $\mathbb P_\Sigma$ with finite conditions which adds  a reflecting sequence for $\Sigma$.
\end{theorem}
The rest of this section is devoted to the proof of \cref{main theorem}.
Suppose $\Sigma$ is an open stationary mapping. Let $X=X_\Sigma$ and $\theta=\theta_\Sigma$.
\begin{definition}[Forcing Poset]\label{mainforcingSigma}
We let $\mathbb P_\Sigma$ consist of triples $p=(\mathcal M_p,d_p,f_p)$, where
\begin{enumerate}
\item\label{it1-maindef} $\mathcal M_p$ is a finite $\in$-chain of models in ${\rm dom}(\Sigma)$.
\item\label{it2-maindef} $d_p:\mathcal M_p\rightarrow [H_\theta]^{<\omega}$ is a function such that if $M\in N$, then $d_p(M)\in N$.
\item\label{it3-maindef} $f_p$  is a regressive function on $\mathcal M_p$, i.e., for every $M\in\mathcal M_p$, $f_p(M)\in [M]^{<\omega}$, such that whenever 
$P\in M$ are $\mathcal M_p$ and $f_p(M)\in P$, then $P\cap X\in\Sigma(M)$.
\end{enumerate}
We equip $\mathbb P_\Sigma$ with the following ordering. We say $q$ is stronger than $p$ and write $q\leq p$, if 
\begin{enumerate}
    
\item $\mathcal M_p\subseteq \mathcal M_q$.
\item For each $M\in\mathcal M_p$, $d_p(M)\subseteq d_q(M)$.
\item $f_p\subseteq f_q$.
\end{enumerate}
\end{definition}

For the sake of convenience, we let $\varnothing$ be in $\mathcal M_p$, for every $p\in \mathbb P_\Sigma$, and leave $f_p(\varnothing)$ undefined.
\begin{lemma}\label{topMshort}
Suppose $p$ is a condition in $\mathbb P_\Sigma$. Let $M$ be an element of ${\rm dom}(\Sigma)$ containing  $p$. Then 
there is a condition $p^M\leq p$ such that $M\in\mathcal M_{p^M}$.
\end{lemma}
\begin{proof}
We let $p^M$ be defined as follows. Let $\mathcal M_{p^M}$ be just $\mathcal M_p\cup\{M\}$, extend $d_p$ as a function by letting 
$d_{p^M}(M)=\varnothing$, and  also extend $f_p$ as a function
by  letting  $f_{p^M}(M)$ be some finite set in $M\setminus\mathcal \bigcup M_p$. It is easily seen $p^M$ is a condition and extends $p$.
 \end{proof}

\begin{proposition}\label{genericity1}
 Suppose $\theta^*>|\mathcal P(H_\theta)|^+$ is a  regular cardinal.
 Assume $M^*\prec H_{\theta^*}$ is countable and contains $\Sigma$ and $[X]^{\omega}$.
Let $M\coloneqq H_\theta\cap M^*$.
Suppose $p_0\in\mathbb P_\Sigma$ is such that $M\in\mathcal M_{p_0}$.
Then $p_0$ is $(M^*,\mathbb P_\Sigma)$-almost strongly generic.
\end{proposition}
\begin{proof}
Fix $p\leq p_0$.  We aim to show that $S_p$ is $M^*$-stationary, thus suppose $F\in M^*$ is an algebra on $\mathbb P_\Sigma$. Notice that
$H_\theta=\bigcup{\rm dom}(\Sigma)\in M^*$, and hence $\mathbb P_\Sigma\in M^*$.
Let $p\rest{M}=(\mathcal M_p\cap M,d_p\rest_{M},f_p\rest_{M})$.  It is clear that $p\rest{M}$ is a condition in $M$.  Set
$$\mathfrak{X}=\{(P,Q)\in\mathcal M_p\times\mathcal M_p: P\subseteq M\subseteq Q \text{ and } f_p(Q)\in P\in Q\}.$$ 
 If $(P,Q)\in\mathfrak{X}$, then there is a finite set $b^Q_P$ in $P$, and hence in $M$, such that 
 $$[b^Q_P,P\cap X]\subseteq\Sigma(Q).$$
  Fix such sets. Set $b_P=\bigcup\{b^Q_P:(P,Q)\in\mathfrak{X}\}$
 and
 $B=\{b_P: P\in {\rm dom}(\mathfrak X)\}$. Pick some regular cardinal $\mu$ with $|H_\theta|<\mu<\theta^*$ and consider 
 the following set which is easily verified to be an unbounded subset of  $[X]^\omega$ belonging to $M^*$.
 $$\mathcal U=\{R\cap X: \{B, [X]^{\omega},F,p\rest{M},\Sigma\}\subseteq R\prec H_\mu\text{ and } R \text{ is countable }\}.$$
 Since $X\in H_\theta$, every algebra on $X$ belongs to $H_\theta$, and hence $\Sigma(M)$ is also $M^*$-stationary.
 Now by the $M^*$-stationarity of $\Sigma(M)$ and \cref{unboundedness 2}, one can find  $A\in \mathcal U\cap M^*\cap \Sigma( M)$. 
Since  $M^*\prec H_{\theta^*}$, there is $R\in M^*$ with $\{B, [X]^\omega,F,p\rest{M},\Sigma\}\subseteq R$ such that
 $A=R\cap X$. Fix such $R$.
 Notice that as before, $X$ and $H_\theta$ are in $R$ as $[X]^\omega$ and $\Sigma$ belongs to $R$.
 Now using the openness of $\Sigma( M)$, there exists a finite set
 $b^M_R \subseteq A$ such that the interval $[b^M_R,A]$ is included in $\Sigma( M)$. Set $b_R=b^M_R\cup b_M$.
 We  need to  extend  $p\rest{M}$ to a condition in $R$ so that its extensions in $R$ are compatible with $p$. To this end,
  let  $p^*$ be the same as $p\rest{M}$ except about $d_{p^*}$, where we let it be defined on $\mathcal M_p\cap M$ by $d_{p^*}(P)=d_p(P)\cup b_{P^+}$,
where $P^+$ is the next model of $P$ in $(\mathcal M_p\cap M)\cup\{R\}$.
 It is clear that $p^*$ is a condition  in $R$ extending $p\rest_M$, since $p\rest{M},B, b_R\in H_\theta\cap R$.
   Now, suppose that $q\in R$ extends $p^*$. We shall show that $q$ is compatible  with $p$. Put $\mathcal M_r=\mathcal M_q\cup\mathcal M_p$.
 Let also $d_r$ be defined  on $\mathcal M_r$ as follows $d_r(P)=d_q(P)$ if $P\in M$, and $d_r(P)=d_p(P)$ otherwise.
  We simply put $f_r=f_p \cup f_q$. We shall show that $r$ is a condition.  \cref{it1-maindef,it2-maindef} of \cref{mainforcingSigma} are clearly satisfied by $r$, and what remains to be shown is \cref{it3-maindef}, i.e.,  whenever $ P\in Q$ are in $\mathcal M_r$ with  $f_r(Q)\in P$, we have that
  $P\cap X\in\Sigma(Q)$.
  To avoid the trivial cases, we may assume
  that $Q\notin M$, 
  and $P\in\mathcal M_q\setminus\mathcal M_p$. In this case
  there is $S\in\mathcal M_p\cap M$ such that $S\in P\in S^+$, where $S^+$ is as above. Now, if $S^+\in R$, then   $(S^+,Q)$ belongs to  $\mathfrak{X}$, and on the other hand $b_{S^+}^Q\subseteq b_{S^+}\subseteq d_q(S)\in P$. Consequently, we have that
 $$P\cap X\in [b_{S^+}^Q,S^+\cap X]\subseteq \Sigma(Q).$$
 Otherwise, $S^+=R$; in this case if $Q=M$, then
$P\cap X\in [b_R,A]\subseteq \Sigma(M)$, and if $M\neq Q$, then $$P\cap X\in [b_R,M\cap X]\subseteq[b^Q_M, M\cap X]\subseteq \Sigma(Q).$$
Therefore, $r$ is a condition. It is clear that it  extends both $p$ and $q$. Consequently,
$R\cap\mathbb P_\Sigma$ is in $M^*\cap S_p$. On the other hand,
 $R\cap\mathbb P_\Sigma$ is closed under $F$ as
 $F$ is in $R$.

\end{proof}

\begin{corollary}\label{MRP-almost}
$\mathbb P_\Sigma$ is a.s.p.
\end{corollary}
\begin{proof}
It is clear from \cref{topMshort} and \cref{genericity1}.
\end{proof}
For a $V$-generic filter $G\subseteq\mathbb P_\Sigma$, we let   
$\mathcal M_G=\{M:\exists p\in G \text{ such that } M\in\mathcal M_p\}$.
\begin{lemma}\label{cont}
Suppose $G$ is a $V$-generic filter on $\mathbb P_\Sigma$. Then $\mathcal M_G$ is a continuous $\in$-chain.
\end{lemma}
\begin{proof}

We show that  if $( M_n)_{n<\omega}$ is a sequence in $\mathcal M_G$, then $\bigcup_{n<\omega} M_n$
is also in $\mathcal M_G$.
This is equivalent to saying that if $M\in\mathcal M_G$ is not the minimal member and is such that for every $P\in\mathcal M_G$ below $M$,
there is a model in $\mathcal M_G$ between $P$ and $M$ (let us call this property \emph{potentially limit}), then $M$ is the union of models below $M$ in $\mathcal M_G$, i.e., $M$ is a limit model of $\mathcal M_G$. 
Thus suppose a countable model $M\prec H_\theta$ is forced, by a condition $p\in G$, to be a potentially limit model in $\mathcal M_{\dot G}$, where  $\mathcal M_{\dot G}$ is a canonical $\mathbb P_\Sigma$-name associated to $\mathcal M_G$. 
Without loss of generality, we may assume that
$\mathcal M_p$ contains $M$ and also some model below $M$ in $\mathcal M_p$ (this is possible as $M$ is potentially limit). Now working in $V$, let $x\in M$.
If $q\leq p$ is an arbitrary condition, then $q$  forces  $M$ to be a potentially limit model of $\mathcal M_{\dot G}$, and that one can extend $q$ to a condition $q_x$ such that
$x\in d_{q_x}(Q_x)$ where $Q_x$ is the largest model below $M$ in $\mathcal M_{q_x}$. This is possible as $\mathcal M_p\subseteq\mathcal M_q$  and that there is some model below $M$ in $\mathcal M_p$.
It then  implies that any extension of $q_x$ which has a model above $Q_x$  should contain $x$, and 
it shows that the set of conditions $r$ such that $x$  belongs to the decoration $d_r(-)$ of some model in $\mathcal M_r \cap M$ is dense below $p$. 
This together with the fact that $p$ forces $M$ to be a potentially limit model in $\mathcal M_{\dot G}$ imply that $p$ forces  there to be a model below $M$ in $\mathcal M_{\dot G}$ containing $x$.
Therefore, $M$ is the union of its predecessors  in $\mathcal M_G$.
In other words, $M$ is a limit model of $\mathcal M_G$.
\end{proof}

Let $f_G$ be defined on $\mathcal M_G$ by letting $f_G(M)=f_p(M)$ for some, or equivalently all, $p\in G$ with $M\in\mathcal M_p$.

\begin{proposition}\label{adds reflecting sequence}
 $\mathbb P_\Sigma$ adds a reflecting sequence for $\Sigma$. 
\end{proposition}
\begin{proof}
Let $G$ be a $V$-generic filter on $\mathbb P_\Sigma$. By  \cref{cont}, $\mathcal M_G$ is a continuous $\in$-chain of models.
Let $\langle M_\xi:\xi<\omega_1\rangle$ be the natural enumeration of $\mathcal M_G$, i.e., for each $\xi<\omega_1$,
 $M_\xi\in M_{\xi+1}$ . We claim that this is a reflecting sequence for $\Sigma$.
If $\xi<\omega_1$ is a limit ordinal, then $f_G(M_\xi)$ belongs to $M_\xi$, and thus by the continuity, there is
$\zeta<\xi$ so that $f_G(M_\xi)\in M_\zeta$, and thus for each $\eta\in\xi\setminus\zeta$, $f_G(M_\xi)\in M_\eta$.
It is enough to pick $q\in G$ such that $M_\xi,M_\eta\in\mathcal M_q$, and hence $M_\eta\cap X\in \Sigma(M_\xi)$.

\end{proof}

\section{Conclusion}

In  \cite{TodorcevicPFA}, Todorčević introduced and studied the forcing axiom $\rm PFA(S)$, and showed its consistency as well. This forcing axiom states that $S$ is a coherent Suslin tree, and if $\mathbb P$ is a proper forcing which preserves the c.c.c-ness of $S$, and that if $\mathcal D$ is a collection of dense subsets of $\mathbb P$ with $|\mathcal D|\leq\aleph_1$, then there is a $\mathcal D$-generic filter on $\mathbb P$. Now by \cref{asp ccc} alone, $\mathbb P_\Sigma$ preserves the c.c.c-ness of $S$.  Using \cref{MRP-almost}, \cref{adds reflecting sequence}, and standard arguments, one can show that $\rm MRP$ holds under $\rm PFA(S)$.

\section*{Acknowledgment}
Part of this note was discovered when the author was a PhD student  under the supervision of Boban Veličković, in particular the notion of almost strong properness was obtained in a  collaboration with him.
The author would like to thank him for his support and his permission to include it here. I would also like to thank the referee for their meticulous reading, useful suggestions and comments.

\bibliographystyle{elsarticle-num}

\bibliography{Almost_Strong_Properness}
\vspace{1.5cm}
\end{document}